\documentclass[reqno]{amsart}

\usepackage[foot]{amsaddr} %to have authors' address on titlepage 

\usepackage[pagewise]{lineno}\linenumbers
\nolinenumbers
\numberwithin{equation}{section}
\usepackage{latexsym}
\usepackage{amsmath,hyperref}
\usepackage{mathscinet}
\usepackage{enumitem}
\usepackage{amssymb}
\usepackage{mathrsfs}
\usepackage{graphicx,colordvi,psfrag}
\usepackage{upgreek}
\usepackage{ifthen}
\usepackage{color}

\usepackage[T1]{fontenc}
\usepackage[latin1]{inputenc}

\numberwithin{equation}{section}

\newtheorem{definition}{Definition}[section]
\newtheorem{theorem}[definition]{Theorem}

\newtheorem{remark}[definition]{Remark}

\usepackage{latexsym}
\usepackage{amsmath}
\usepackage{amssymb}

\newcommand{\R}{{\mathbb R}}
\newcommand{\N}{{\mathbb N}}
\renewcommand{\P}{\mathbb P}
\providecommand{\norm}[1]{\| #1 \|}%{\left\vert #1 \right\vert}

\renewcommand{\div}{\mathrm{div}\,}

\newcommand{\n}{\mathrm{n}}

\allowdisplaybreaks %to allow page breaks in long equations

\title[Continuous dependence of Navier-Stokes equations]{A note on continuous dependence of Navier-Stokes equations with oscillating force}
\author[A. Dutta]{Anirban Dutta}
\date{\today}
\address{Queen's University}
\email{21ad53@queensu.ca}
\begin{document}

%\subjclass[2010]{Primary: }
% 35K58 Semilinear parabolic equations 
% 35B35 Stability 
% 35B40 Asymptotic behavior of solutions 
% 35Q30 Navier-Stokes equations
% 35Q35 PDEs in connection with fluid mechanics 

%Normally stable, normally hyperbolic, global existence, critical spaces, fluid-solid interactions, rigid body motion.

\begin{abstract}
In this paper, we examine the averaging effect of a highly oscillating external force on the solutions of the Navier-Stokes equations. We show that, as long as the force time-average decays over time, if the frequency and amplitude of the oscillating force grow, then the corresponding solutions to Navier-Stokes equations converge (in a suitable topology) to the solution of the homogeneous equations with same initial data. Our approach involves reformulating the system as an abstract evolution equation in a Banach space, and then proving  continuous dependence of solutions on both initial conditions and external forcing.
\end{abstract}

\maketitle

\maketitle

\section{Introduction}\label{sec1}
In this paper, we show the averaging effect of a highly oscillating external force on %demonstrate that an oscillating external force with a growing amplitude has minimal impact on the 
solutions of a d-dimensional Navier-Stokes equation.
%, as outlined below, which closely resembles the model studied by Chepyzhov, Pata, and Vishik in their paper \cite{MR2475550}.
Following Chepyzhov, Pata, and Vishik  \cite{MR2475550}, we consider a bounded domain $\Omega \subset \R^d$, $d=2,3$ with $C^3$ boundary and $T>0$. For each $n\geq1$, consider the following system:
\begin{equation}\label{intro1}
\left\{
\begin{aligned}
   &\nabla\cdot v_n %\partial_{x_1}v_n^{1}+\partial_{x_2}v_n^2
   =0\qquad&&\text{in }(0,T)\times\Omega,
   \\
   &\partial_t v_n+ (v_n\cdot \nabla)v_n % v_n^1\partial_{x_1}v_n+v_n^2\partial_{x_2}v_n
   =\nu\Delta v_n-\nabla P+n^{\frac{\rho}{p}}g(nt,x) \qquad&&\text{in } (0,T)\times \Omega,
   \\
   &v_n=0\qquad&&\text{on }(0,T)\times \partial \Omega, 
\end{aligned}
\right.    
\end{equation}
where $v_n(t,x)$ denotes the fluid velocity at time $t$,  $\nu$ is the coefficient of kinematic viscosity, and $P=P(t,x)$ is the (unknown) pressure. The fluid density is assumed to be $1$. 
%The operators $\nabla$ and $\Delta$ %in \eqref{intro1}
%refer to the space variable $x\in \Omega$.
 In \eqref{intro1}, $\rho\in[0,1)$ is a fixed scaling parameter, and $p\in(1,\infty)$.
We assume that the function $g:\R^+\times \Omega\to\R^d$ is 
such that 
\begin{equation}\label{intro-g}
    \left(\frac{1}{T}\right)^{1-\rho }\int_0^T \norm{ g(s,x) }^p ds\stackrel{T\to\infty}{\longrightarrow} 0
\end{equation}
 in a suitable functional space. As $n$ grows, the term $n^{\rho/p}g(nt,x)$ represents an external force with increasing frequency and amplitude, while its time average decreases according to equation \eqref{intro-g}. Such averaging helps the solution $v_n$ of \eqref{intro1} converges (in a suitable functional space) to the solution $v$ of the system
\begin{equation}\label{intro2}
\left\{
\begin{aligned}
   &\nabla\cdot v %\partial_{x_1}v_n^{1}+\partial_{x_2}v_n^2
   =0\qquad&&\text{in }(0,T)\times\Omega,
   \\
   &\partial_t v+ (v\cdot \nabla)v 
   =\nu\Delta v-\nabla P \qquad&&\text{in } (0,T)\times \Omega,
   \\
   &v=0\qquad&&\text{on }(0,T)\times \partial \Omega, 
\end{aligned}
\right.    
\end{equation}
where the external force is zero. For the sake of exposition, we may assume that the initial conditions in \eqref{intro1} and \eqref{intro2} are the same, although our methods allow to relax this assumption. In \cite{MR2475550}  the authors show that the uniform  global attractor of \eqref{intro1} converges as $n\to\infty$ to that of \eqref{intro2}. In their analysis, all the functional spaces are assumed to be Hilbert  spaces and only the case $p=2$ is considered in \eqref{intro1}.  
 
To establish our result, we rewrite the above equations as ordinary differential equation (ODE) on a (infinite-dimensional) Banach space. We consider the abstract evolution equation
%This paper presents findings regarding the continuous dependence of solutions to the following ordinary differential equations (ODE) within the context of Banach spaces.
\begin{equation}\label{intro-ODE-n}
\begin{cases}
\displaystyle
\frac{d\phi_n(t)}{dt} + A\phi_n(t) = f_n(t, \phi_n(t)), \qquad & t>t_{0} \\[10pt]  
\displaystyle
\phi_n(t_{0})= u^n_{0}.
\end{cases}
\end{equation}
In \eqref{intro-ODE-n} the operator $A$ is  linear, while  the map $x\mapsto f_n(t,x)$ is a nonlinear time-dependent forcing.
We prove that, in a suitable topology, if $f_n$ and $u^n_{0}$ converge to some $f$ and $u_0$, respectively, as $n\to\infty$, then the solution $\phi_n$ of \eqref{intro-ODE-n}  converges to the solution $\phi$ of the  ODE
\begin{equation}
\begin{cases}
\displaystyle
\frac{d\phi(t)}{dt} + A\phi(t) = f(t, \phi(t)), \qquad & t>t_{0} \\[10pt]  

\displaystyle
\phi(t_{0})= u_{0}.
\end{cases}
\end{equation}
This result can phrased as the continuous dependence of the solution of an abstract ODE upon initial condition and forcing.

Continuous dependence of solutions to abstract ODEs has been studied in various contexts. For example, Henry (\cite[Theorem 3.4.1]{henry}) investigates the case when  $A$ is a sectorial operator, showing the continuous dependence of mild solutions upon both the initial condition and the forcing term. Furthermore, in \cite[Theorem 3.4.9]{henry}, the author explores another concept of continuous dependence, which leads to the ``method of averaging''. For quasilinear parabolic evolution equations, K\"ohne, Pr\"uss, and Wilke (\cite[Theorem 2.1]{MR2643804}) establish continuous dependence solely on the initial conditions, under the assumption that \( A \) possesses the property of maximal \( L^{p} \)- regularity.  As a result, their continuous dependence results apply to strong solutions.\\

In this paper, we will establish two main abstract theorems on continuous dependence (Theorems \ref{main theorem1} and \ref{4.2}).
In both results we consider strong solutions and establish continuous dependence upon both the initial condition and the forcing term, assuming that \( A \) possesses the property of maximal \( L^{p} \)- regularity. In Theorem 
\ref{main theorem1} we consider a more general form of forcing than the one Henry considers in (\cite[Theorem 3.4.1]{henry}) and, unlike \cite{MR2643804},  we consider non-autonomous ODEs. Theorem \ref{4.2} on the "method of averaging" for abstract ODEs yields Theorem \ref{main theorem3} about continuous dependence for Navier-Stokes equations \eqref{intro1}.

%\begin{remark}
%We would like to clarify that Theorem \eqref{main theorem1} is not directly concerned with the oscillating external force. Instead, it is presented for its independent interest and significance within the scope of our work.
%\end{remark}

The paper is organized as follows.
%\begin{itemize}
%\item i
In Section \ref{sec:notation} we introduce the relevant notations and definitions, and recall some useful fundamental results.
In Section \ref{sec:cont_dep} we first prove Theorem \ref{main theorem1} and discuss how our analysis differs from Henry's approach. Afterwards, we prove Theorem \ref{4.2}. %, which forms the core of Theorem \eqref{main theorem3}. 
Section \ref{application} is dedicated to the proof of Theorem \ref{main theorem3}, which follows from Theorem \ref{4.2}.

\section{Notation and useful results}\label{sec:notation}
Let $\Omega \subseteq \mathbb{R}^d$ be a domain with $d \geq 1$, and let $p \in [1, \infty)$. We denote by $L^p(\Omega)$ the usual Lebesgue space, equipped with the norm  
\[
\|w\|_{p,\Omega} := \left( \int_\Omega |w|^p\, dx \right)^{1/p},
\]
and by $W^{k,p}(\Omega)$ the Sobolev space of functions with weak derivatives up to order $k \in \mathbb{N}$ in $L^p(\Omega)$, endowed with the norm  
\[
\|w\|_{W^{k,p}(\Omega)} := \left( \sum_{|\alpha| \leq k} \|D^\alpha w\|_{p,\Omega}^p \right)^{1/p}.
\]
Here, $|\cdot|$ denotes the Euclidean norm for vector fields, and the same symbol is used for the Frobenius norm in the case of tensor fields. The notation $D^\alpha w$ refers to the weak derivative of order $\alpha$ of the function $w$.

For non-integer values $k \notin \mathbb{N}$, we define the Sobolev space via Besov spaces, i.e., $W^{k,p}(\Omega) := B^k_{pp}(\Omega)$. We recall the following characterization of Besov spaces:
\[
B^s_{qp}(\Omega) = (H^{s_0}_q(\Omega), H^{s_1}_q(\Omega))_{\theta, p},
\]
where the right-hand side denotes the real interpolation between Bessel potential spaces $H^{s_0}_q(\Omega)$ and $H^{s_1}_q(\Omega)$, with $s_0 \neq s_1 \in \mathbb{R}$, $p, q \in [1, \infty)$, $\theta \in (0,1)$, and $s = (1 - \theta)s_0 + \theta s_1$. Moreover, Bessel potential spaces satisfy the complex interpolation identity
\[
H^s_q(\Omega) = [H^{s_0}_q(\Omega), H^{s_1}_q(\Omega)]_\theta.
\]

In the Hilbertian case, we have the identification
\[
W^{s,2}(\Omega) = B^s_{22}(\Omega) = H^s_2(\Omega) =: H^s(\Omega),
\]
as discussed in \cite{bergh}. Finally, for every $s > 1/q$, we define the subspace of functions vanishing on the boundary as
\[
{}_0H^s_q(\Omega) := \{ f \in H^s_q(\Omega) : f = 0 \text{ on } \partial \Omega \}.
\]

Let $(X, \norm{\cdot}_X), (Y, \norm{\cdot}_Y)$ be two Banach spaces. For an interval $I\subset\R$ and $1\le p<\infty$, $L^p(I;X)$ (resp. $W^{1,p}(I;X)$, $k\in \N$) denotes the space of strongly measurable functions $f$ from $I$ to $X$ for which $\left(\int_I \norm{f(t)}^p_X\; d t\right)^{1/p}<\infty$ (resp. $\sum^1_{\ell=0}\left(\int_I \norm{\partial^\ell_t f (t)}^p_X\; d t\right)^{1/p}<\infty$).
We  denote by $L^p_{\mathrm{loc}}(I;X)$ the set of all functions $f:I\to X$ belonging to $L^p(K;X)$ for every compact sub-interval $K\subset I$.
As customary, we use the notation $\norm{f}_{L^p(I;Y)\cap W^{1,p}(I;X)}:=\norm{f}_{L^p(I;Y)}+\norm{f}_{W^{1,p}(I;X)}$. $C(I;X)$ denotes the space of continuous functions $f:I\to X$ %that are $k$ times c differentiable with values in $X$, and have 
such that $\norm{f}_{C(I,X)}:=\displaystyle\sup_{t\in I}\norm{f(t)}_X < \infty$.

Let \( A: D(A) \subseteq X \to X \) be a closed, densely defined linear operator. Let \( X_{p} := (X, D(A))_{1-\frac{1}{p}, p} \) be the (real) interpolation space, see e.g. Chapter $1$ in \cite{interpolation_theory}. In this case, one can show that the norm on $X_p$ can be expressed as %, equipped with the norm 
\[
\norm{f}_{X_p} := \inf \left\{ \norm{F}_{L^p(0,\infty;X) \cap W^{1,p}(0,\infty;D(A))} \colon f = F(0) \right\}.
\]
This definition makes $X_p$ the natural choice for the space of initial data for the abstract ODE \eqref{eq:def_max_Lp} we consider below.

\begin{definition} Let $t_0\in\R$, $0<T\leq \infty$, and set $I=[t_0,t_0+T)$. For $p\in(1,\infty)$, we say that $A$ \emph{has the property of maximal $L^p$-regularity on $[t_0,t_0+T)$} if and only if there exists a constant $C=C(T)>0$ such that for every $f\in L^p(I; X)$ and every $u_0\in X_p$ there exists a unique $u \in  L^{p}(I;D(A)) \cap W^{1,p}(I;X)$ satisfying the non-homogeneous abstract ODE
\begin{equation} \label{eq:def_max_Lp}
\begin{cases}
\displaystyle
\frac{du}{dt} + Au = f(t), \qquad & t_0\le t\le t_0+T, \\[5pt]  
\displaystyle
u(t_0)=u_0 %\in X_p
\end{cases}
\end{equation}
for almost every $t\in I$ and such that 
\begin{align*}
     \norm{u}_{W^{1,p}(I; X) \cap L^{p}(I; D(A))} \leqslant C(T)\!\left(\norm{ u_0}_{X_p} +\norm{f}_{L^{p}(I;X)}\right).
\end{align*}     
     
\end{definition}

Note that the above constant $C$ can be chosen as a non-decreasing function of $T$. 

\begin{remark}The operator 
    \( A \) has the property of maximal \( L^p \)-regularity \textit{only if} $A$ is a sectorial operator, see \cite[Proposition 3.5.2]{Prussbook}. However, the converse does not generally hold and  %The proof of this result can be found in \cite[Proposition 3.5.2]{Prussbook}, 
    a counterexample is provided by Coulhon and Lamberton in \cite[Theorem 2.1]{Merdy}. When \( X \) is a Hilbert space,
one can show that  $A$ is sectorial if and only if A has the property of maximal \( L^p \)-regularity, see \cite[Theorem 3.5.7]{Prussbook}.

\end{remark}

% {For a domain $\Omega\subseteq\R^d$, $d\geq1$, and $p\in [1,\infty]$, we denote with $L^p(\Omega)$ and $W^{k,p}(\Omega)$ the Lebesgue and Sobolev spaces endowed with the norms 
% \[
% \norm{w}_{p,\Omega}:=\left(\int_\Omega |w|^p\;d x\right)^{1/p}\quad\text{and}\quad 
% \norm{w}_{W^{k,p}(\Omega)}:=\left(\sum_{|\alpha|\le k}\norm{D^\alpha w}^p_{p,\Omega}\right)^{1/p},
% \]
% respectively. In the above, we have denoted by $|\cdot|$ the Euclidean norm of vector fields (we will use the same symbol for the Frobenius norm of tensor fields). Also, $D^\alpha$ denotes the $\alpha$-th order weak derivative of $w$. If $k\notin \N$, then $W^{k,p}(\Omega):=B^k_{pp}(\Omega)$. We recall the following characterization of Besov spaces $B^s_{qp}(\Omega)=(H^{s_0}_q(\Omega),H^{s_1}_q(\Omega))_{\theta,p}$ as real interpolation of Bessel potential spaces, and of Bessel potential spaces $H^s_q(\Omega)=[H^{s_0}_q(\Omega),H^{s_1}_q(\Omega)]_{\theta}$. These characterizations are valid for $s_0\ne s_1\in \R$, $p,q\in [1, \infty)$, $\theta\in (0,1)$ and $s=(1-\theta)s_0+\theta s_1$. We also recall that $W^{s,2}(\Omega) =B^s_{22}(\Omega) = H^s_2(\Omega)=:H^s(\Omega)$ (see \cite{bergh}). For every $s>1/q$, 
% \[
% {_0}H^s_q(\Omega):=\{f\in H^s_q(\Omega):\;f=0\text{ on }\partial \Omega\}.
% \]}

Throughout the paper, we write ``$\lesssim\,$'' to denote ``$\leq C$\,'' for some uniform constant $C$ (independent of the various parameters in the problem being considered) that may vary from line to line.
\section{Main abstract theorems on continuous dependence for strong solutions}\label{sec:cont_dep}
We are now ready to present new results on the continuous dependence of strong solutions on both the initial conditions and the forcing term.

\begin{theorem}\label{main theorem1}
  Let $p> 1$. Suppose that $A$ has the property of maximal $L^{p}$-regularity. For every $n\geq0$, consider a function $f_n:[0,\infty)\times X_p\to X$ such that
 \begin{enumerate}[label=\rm{(I.\alph*)},ref=\rm{(I.\alph*)}]
       \item \label{existence_reason1} %For every $K\subset\subset [0,\infty)$ we have 
      $f_n(\cdot, 0) \in L^{p}_{\mathrm{loc}}([0,\infty); X)$,
      \item \label{existence_reason2} %For every $K\subset\subset [0,\infty)$ and  
       for all \( R > 0 \), there exists a function \( \Psi^n_R(t) \in L^p_{\mathrm{loc}}([0,\infty)) \) such that for any \( u_1, u_2 \in X_p \) with \( \| u_1 \|_{X_p}, \| u_2 \|_{X_p} \leq R \), the following inequality holds:$$
    \| f_n(t, u_1) - f_n(t, u_2) \|_X \leq \Psi^n_R(t) \| u_1 - u_2 \|_{X_p}.$$

  \end{enumerate}
  Suppose that the sequence $(f_n)_{n\geq1}$ converges to $f_0$ in $X$ uniformly on bounded subsets of $[0,\infty)\times X_p$, that is, 
  \begin{enumerate}[label=\rm{(II.\alph*)},ref=\rm{(II.\alph*)}]
      \item for every $\varepsilon>0$, every compact sub-interval $K\subset [0,\infty)$, and every $\eta>0$, there exists $N\geq1$ such that 
      \[
      \displaystyle\sup_{t\in K}\sup_{\norm{u}_{X_p}\leq \eta}\norm{f_n(t,u)-f_0(t,u)}_X<\varepsilon\]
      for every $n\geq N$.
  \end{enumerate}
  Let $(u_0^n)_{n\geq 0}\subset X_p$ be such that 
  \begin{enumerate}[label=\rm{(III.\alph*)},ref=\rm{(III.\alph*)}]
      \item $\norm{u_0^n-u_0^0}_{X_p}\to0$ as $n\to\infty$,
  \end{enumerate}
   %$\{f_{n}\}_{n=0}^\infty$ be a sequence of functions defined on $X_{p}$ into $X$ such that \\
  % $f_{n}(\cdot, 0) \in L^{p}_{\mbox{\tiny{loc}}}([0,\infty); X)$ and 
%   each $f_{n}$ locally Lipschitz (i.e. 
%for each $R > 0$ and\\
   %$n \in \mathbb {N}\cup \{0\}$, 
%   there exist $\Psi^{n}_{R}\in L^{p}_{\mbox{loc}}([0,\infty))$ such that for all $u_{1}, u_{2} \in X_{p}$ satisfying $\norm{ u_{1}} _{X_{p}}, \norm{ u_{2}} _{X_{p}} \leq R$ and a.a. $t \in [0,a_{0}]$, $\norm{ f_{n}(t, u_{1})- f_{n}(t, u_{2})} _{X} \leq \Psi^{n}_{R}(t) \norm{ u_{1}- u_{2}} _{X_{p}}$) and $f_{0}(t, x) = \lim_{n\to\infty} f_{n}(t, x)$ uniformly for $(t,x)$ in a neighborhood of any point of $\mathbb{R^{+}} \times X_{p}$ (i.e. for all $\delta > 0$ , $\norm{ f_{n}(t, u)- f_{0}(t, u)} _{X} \rightarrow 0$ uniformly for $\norm{ u} _{X_{p}} \leq \delta $ for a finite time interval). Also, assume that $\norm{ u^{n}_{0}- u^{0}_{0}}_{X_{p}} \rightarrow 0 $  as $n \rightarrow \infty$.\\
   and for each $n\geq0$, let $\phi_{n}$ be the %maximally defined 
   solution of 
\begin{equation}
\begin{cases}
\displaystyle
\frac{d\phi_{n}(t)}{dt} + A\phi_{n}(t) = f_{n}(t, \phi_{n}(t)), \qquad & t>t_{0} \\[10pt]  
\phi_{n}(t_{0})= u^{n}_{0}.  %\in X_{p}
\end{cases}
\end{equation}
Suppose that $\phi_n$ is defined on  the maximal interval of existence $[t_0,t_0+T_n)$. Then $T_n\in (0,\infty]$,    % is the largest such extended positive real number.\\
%which exists on $[t_{0},  t_{0}+ T_{n})$. 
\begin{align}\label{T_0leq_limsup}
T_{0} \leq \limsup_{n\to\infty}{T_n}, 
\end{align}
and the sequence $(\phi_n)_{n\geq1}$ converges to $\phi_0$ in the following sense: for every compact sub-interval $K\subset [t_0, t_0+T_0)$ we have
\begin{itemize}
    \item[(a)]  $\displaystyle\sup_{t\in K}  \norm{ \phi_{n}(t)- \phi_{0}(t)}_{X_{p}} \longrightarrow 0$ as $n\to\infty$, and
    \item[(b)]  $\displaystyle\norm{ \phi_{n}- \phi_{0}} _{W^{1,p}(K; X) \cap L^{p}(K; D(A))}\longrightarrow 0$ as $n\to\infty$.
\end{itemize}
%$T_{0} \geq \limsup_{n\to\infty}{T_nX}$ and 
%\begin{align*}
%    \norm{ \phi_{n}(t)- \phi_{0}(t)}_{X_{p}} \rightarrow 0
%\end{align*}
%uniformly on the compact sub-interval $K$ of $[t_{0},  t_{0}+ T_{0})$ and 
%\begin{align*}
%    \norm{ \phi_{n}(t)- \phi_{0}(t)} _{W^{1,p}(K; X) \cap L^{p}(K; D(A))} \rightarrow 0.
%\end{align*}
\end{theorem}
\begin{proof}
%    Suppose that the claim is not true. Therefore, there exists $\delta > 0$ and $t_{n_k} \in [t_0, t_0+ T_0)$ such that for all $t \in [t_0, t_{n_k})$, $\norm{ \phi_{n_k}(t) - \phi(t)} _{X_P} < \delta$ and $\norm{ \phi_{n_k}(t_{n_k}) - \phi(t_{n_k})} _{X_P} = \delta$ (due to the continuity of $u$ and $u_{n_k}$. Now by using the fact that  
 \ref{existence_reason1} and \ref{existence_reason2} ensure that for all $n \in \N \cup\{0\}$,  we have $T_n > 0$, see \cite[Theorem 3.1.]{Bari}. Letting $T'_0 < T_0$, there exists $R> 0$ such that $\norm{\phi_0(t)}_{X_p} \leq R$, for all $t \in [t_0,t_0+T'_0]$. % Let us 
We claim that for large enough $n\in \N$ and for all $t \in [t_0,t_0+T'_0]$, we have the bound $$ \norm{\phi_n(t)}_{X_p} \leq 2R. $$ %We know that for all $n\in \N$, there exists  $t_n \in (0,T_n)$ such that $t_n \leq T_0'$ and the inequality $\norm{\phi_n(t)}_{X_p} \leq 2R$ holds for all $t \in [t_0,t_0+t_n]$.
Define
\begin{align}
        t_{n} :=  \sup \{t\in (0, T_n) \colon t \leq T'_0, \norm{\phi_n(\tau)}_{X_p} \leq 2R, \mbox{ for all }\tau \in [0,t]\},\label{def-t_n}
\end{align}
and to prove our claim,  we show that 
% there exists $\delta > 0$ (independent of $n$) such that $t_n > \delta $. Then, by induction, we conclude that for large enough $n\in \N$, we have 
$t_n = T'_0 $. 

Since $A$ has the property of maximal $L^{p}$- regularity  on $[t_0,t_0+T'_0]$,  we have that, for all $t_n' \leq t_n (\leq T'_0)$,
\begin{equation}\label{Maximal_regularity}
    \begin{aligned}
    &\norm{ \phi_{n} - \phi_{0} }_{W^{1,p}(t_0, t_0+t'_n}; X) \cap L^{p}(t_0, t_0+t'_n; D(A))\\
    & \lesssim  \norm{ u^{n}_{0}- u^{0}_{0}}_{X_{p}} + \norm{ f_{n}(\cdot, \phi_{n}(\cdot)) - f_0(\cdot, \phi_{0}(\cdot))}_{L^{p}(t_0, t_0+t'_n}; X)\\
    &\lesssim  \norm{ u^{n}_{0}- u^{0}_{0}}_{X_{p}} +\Delta_{n} {t'_n}^\frac{1}{p}+ \norm{ \Psi^0_2R}(\phi_{n} - \phi_{0})_{L^{p}(t_0, t_0+{t'_n}; X_p)}
    \end{aligned}
\end{equation}
where $\Delta_{n} := \sup \{\norm{ f_{n}(s, \phi_{n}(s)) - f_{0}(s, \phi_{n}(s))} _{X} \colon t_0\leq s \leq t_0+{t'_n, \norm{x}_{X_p} \leq 2R}  \}$. We know that the following inequality holds (see \cite[Corollary 1.14.(ii)]{interpolation_theory}):
\begin{align}\label{embedding}
    \sup_{t \in [t_0, t_0+t_n']} \norm{ \phi_{n}(t) - \phi_{0}(t) }_{X_p} \lesssim \norm{ u^{n}_{0}- u^{0}_{0}}_{X_{p}}+\norm{ \phi_{n} - \phi_{0} }_{W^{1,p}(t_0,t_0+{t'_n}; X) \cap L^{p}(t_0,t_0+{t'_n}; D(A))},
\end{align}
Therefore, by combining \eqref{Maximal_regularity} and \eqref{embedding}, we get the the following estimate: 
\begin{align}\label{Combination_of_maximal_and_embedding}
    \sup_{t \in [t_0, t_0+t_n']} \norm{ \phi_{n}(t) - \phi_{0}(t) }_{X_p} \lesssim \norm{ u^{n}_{0}- u^{0}_{0}}_{X_{p}} +\Delta_{n} {t'_n}^\frac{1}{p}+ \norm{ \Psi^0_{2R}(\phi_{n} - \phi_{0})}_{L^{p}(t_0, t_0+{t'_n}; X_p)}
\end{align}
Since $\Psi^0_{2R} \in L^p(t_0, t_0 + T'_0)$, there exists a constant $0 < \delta \leq {T'_0}$ such that for all $a, b \in (t_0, t_0 + T'_0)$ with $|a - b| < \delta$, one has
\begin{align}\label{constant}
    \|\Psi^0_{2R}\|_{L^p(a,b)} \leq \frac{1}{2c},
\end{align}
where $c>0$ is the implied constant of \eqref{Combination_of_maximal_and_embedding}.  % Since the following estimate holds for all $t \in [t_0, t_0+\delta]$ (see \cite[Corollary 1.14.(ii)]{interpolation_theory}):
% \begin{align}\label{embedding}
%     \norm{ \phi_{n}(t) - \phi_{0}(t) }_{X_p} \lesssim \norm{ u^{n}_{0}- u^{0}_{0}}_{X_{p}}+\norm{ \phi_{n} - \phi_{0} }_{W^{1,p}(t_0,t_0+{\delta}; X) \cap L^{p}(t_0,t_0+{\delta}; D(A))},
% \end{align}
Now, from \eqref{Combination_of_maximal_and_embedding} and \eqref{constant}, we obtain the following estimate for all $t \in [t_0, t_0 + t_n']$:
\begin{equation}\label{9}
    \begin{aligned}
       \|\phi_n(t) - \phi_0(t)\|_{X_p} \lesssim \|u_0^n - u_0^0\|_{X_p} + \Delta_n {T'_0}^{\frac{1}{p}}.
    \end{aligned}
\end{equation}
Therefore, { for all $0 <\epsilon < \frac{R}{2}$, there exists $N \in \N$ such that for all $n \geq N$ we have $\norm{ \phi_{n}(t) - \phi_{0}(t) }_{X_p} < \epsilon$, whenever $t \in [t_0,t_0+\delta)$. So, we ensure that $\delta < T_n$ and therefore, we have $t_n > \delta$. Now, we }proceed with the same analysis for large enough $n$ and  by considering the initial values to be \( \phi_{n}(t_0 + \delta) \) and \( \phi_{0}(t_0 + \delta) \) to show that $t_n > 2\delta$.  Now, by induction, we  show that, for large enough $n\in \N$, we have $t_n = T_0'$. Therefore, for large enough $n\in \N$, the estimate \eqref{9} holds on  \([t_0, t_0+T'_0] \), and 
%Therefore, we conclude that
\begin{align*}
    \sup_{t\in[t_0,t_0+T'_0]} \norm{ \phi_{n}(t) - \phi_{0}(t) }_{X_p} \rightarrow 0\mbox{ as $n \rightarrow \infty$}.
\end{align*}
%uniformly on $[t_0,t_0+T'_0]$ .  
Furthermore, it is clear from \eqref{Maximal_regularity} that 
\begin{align*}
    \norm{ \phi_{n} - \phi_{0} }_{W^{1,p}(t_0,t_0+T'_0; X) \cap L^{p}(t_0,t_0+T'_0; D(A))} \rightarrow 0 \mbox{ as $n \rightarrow \infty$.}
\end{align*}
{ From the previous analysis, we ensure that for all \( k \in \mathbb{N} \), there exists \( n_k \in \mathbb{N} \) such that \( T_0 - \frac{1}{k} \leq t_{n_k}\leq T_{n_k} \). Therefore, we obtain the inequality \eqref{T_0leq_limsup}
%\[
%T_0 \leq \limsup_{n \to \infty} T_n.
%\]
}
\end{proof}
\begin{remark}
     One could compare the above theorem with \cite[Theorem 3.4.1]{henry}. Following the proof of the latter, one would need that 
    %Although we could adopt Henry's approach to prove this result, it turns out that the above theorem can only be proven if 
    $\Psi^{0}_{R}\in L^{p'}_{\mathrm{loc}}([0,\infty))$ with $p' > p$.
\end{remark}
Now, we will prove  our second main result.
\begin{theorem}\label{4.2}
    Let $p> 1$. Suppose that A has the property of maximal $L^{p}$- regularity.  For every  $n\geq0$, consider a function $f_n:[0,\infty)\times X_p\to X$ such that
    \begin{enumerate}[label=\rm{(I.\alph*)},ref=\rm{(I.\alph*)}]
        \item \label{existence_reason1a} $f_0(t,x) = f(x)$, where $f:\; X_p\to X$ is a Lipschitz function { on  bounded subsets of $X_p$}; 
        \item \label{existence_reason2b} For $n\ge 1$,
        \[
        f_n(t,x)= f(x) + n^{\rho/p}g(nt),
        \]
        where 
        %{\color{blue} $f:\; X_p\to X$ is a locally Lipschitz function}, 
        $\rho \in (0,1)$,  and $g \in L^p_{\mathrm{loc}}(0,\infty; X)$ satisfies the following property that %for some $p\in [1,\infty)$ 
        {for all $T>0$, \[
        \left(\frac{1}{nT}\right)^{1-\rho }\int_{nt_0}^{nt_0+nT} \norm{ g(s) }_X ^p ds \rightarrow 0\quad n \rightarrow \infty.
        \]}
    \end{enumerate}
     %Now we will consider $f_0(t,x) = f(x)$. 
     Let $(u_0^n)_{n\geq 0}\subset X_p$ be such that 
  \begin{enumerate}[label=\rm{(II.\alph*)},ref=\rm{(II.\alph*)}]
      \item $\norm{u_0^n-u_0^0}_{X_p}\to0$ as $n\to\infty$,
  \end{enumerate}
   %$\{f_{n}\}_{n=0}^\infty$ be a sequence of functions defined on $X_{p}$ into $X$ such that \\
  % $f_{n}(\cdot, 0) \in L^{p}_{\mbox{\tiny{loc}}}([0,\infty); X)$ and 
%   each $f_{n}$ locally Lipschitz (i.e. 
%for each $R > 0$ and\\
   %$n \in \mathbb {N}\cup \{0\}$, 
%   there exist $\Psi^{n}_{R}\in L^{p}_{\mbox{loc}}([0,\infty))$ such that for all $u_{1}, u_{2} \in X_{p}$ satisfying $\norm{ u_{1}} _{X_{p}}, \norm{ u_{2}} _{X_{p}} \leq R$ and a.a. $t \in [0,a_{0}]$, $\norm{ f_{n}(t, u_{1})- f_{n}(t, u_{2})} _{X} \leq \Psi^{n}_{R}(t) \norm{ u_{1}- u_{2}} _{X_{p}}$) and $f_{0}(t, x) = \lim_{n\to\infty} f_{n}(t, x)$ uniformly for $(t,x)$ in a neighborhood of any point of $\mathbb{R^{+}} \times X_{p}$ (i.e. for all $\delta > 0$ , $\norm{ f_{n}(t, u)- f_{0}(t, u)} _{X} \rightarrow 0$ uniformly for $\norm{ u} _{X_{p}} \leq \delta $ for a finite time interval). Also, assume that $\norm{ u^{n}_{0}- u^{0}_{0}}_{X_{p}} \rightarrow 0 $  as $n \rightarrow \infty$.\\
  and for $n\geq0$, let $\phi_{n}$ be the %maximally defined 
   solution of 
\begin{equation}\label{important}
\begin{cases}
\displaystyle
\frac{d\phi_{n}(t)}{dt} + A\phi_{n}(t) = f_{n}(t, \phi_{n}(t)), \qquad & t>t_{0} \\[10pt]  
\phi_{n}(t_{0})= u^{n}_{0}.  %\in X_{p}
\end{cases}
\end{equation}
{ Suppose that $\phi_n$ is defined on  the maximal interval of existence $[t_0,t_0+T_n)$. Then $T_n\in (0,\infty]$,    % is the largest such extended positive real number.\\
%which exists on $[t_{0},  t_{0}+ T_{n})$. 
\begin{align}\label{T_0_for_2nd_thm_leq_limsup}
T_{0} {\leq} \limsup_{n\to\infty}{T_n}, 
\end{align}} 
and the sequence $(\phi_n)_{n\geq1}$ converges to $\phi_0$ in the following sense: for every compact sub-interval $K\subset [t_0, t_0+T_0)$ we have
\begin{itemize}
    \item[(a)]  $\displaystyle\sup_{t\in K}  \norm{ \phi_{n}(t)- \phi_{0}(t)}_{X_{p}} \longrightarrow 0$ as $n\to\infty$, and
    \item[(b)]  $\displaystyle\norm{ \phi_{n}- \phi_{0}} _{W^{1,p}(K; X) \cap L^{p}(K; D(A))}\longrightarrow 0$ as $n\to\infty$.
\end{itemize}
%$T_{0} \geq \limsup_{n\to\infty}{T_n}$ and 
%\begin{align*}
%    \norm{ \phi_{n}(t)- \phi_{0}(t)}_{X_{p}} \rightarrow 0
%\end{align*}
%uniformly on the compact sub-interval $K$ of $[t_{0},  t_{0}+ T_{0})$ and 
%\begin{align*}
%    \norm{ \phi_{n}(t)- \phi_{0}(t)} _{W^{1,p}(K; X) \cap L^{p}(K; D(A))} \rightarrow 0.
%\end{align*}
\end{theorem}
\begin{proof}
{ \ref{existence_reason1a} and \ref{existence_reason2b} ensure that for all $n \in \N \cup\{0\}$, $T_n > 0$, see \cite[Theorem 3.1.]{Bari}. Letting $T'_0 < T_0$, there exists $R> 0$ such that the inequality $\norm{\phi_0(t)}_{X_p} \leq R$ holds for all $t \in [t_0,t_0+T'_0]$. We define 
\begin{align}
        T'_{n} :=  \sup \{t\in (0, T_n) \colon t \leq T'_0, \norm{\phi_n(\tau)}_{X_p} \leq 2R, \mbox{ for all }\tau \in [0,t]\}.\label{def-T_n}
\end{align}
}
Let us show that $T_n'=T_0'$.
Since $A$ has the property of maximal $L^{p}$- regularity on $[t_0,t_0+T_0']$ and by the properties of $(f_n)_{n\ge 0}$, we get
\begin{equation}\label{main_inrquality}
    \begin{aligned}
    \mbox{ \quad}\norm{ \phi_{n} - \phi_{0} }&_{W^{1,p}(t_0, t_0+T'_{n}; X) \cap L^{p}(t_0, t_0+T'_{n}; D(A))} \\
    &\lesssim  \norm{ u^{n}_{0}- u^{0}_{0}}_{X_{p}} 
    + \norm{ f( \phi_{n}(\cdot)) - f(\phi_{0}(\cdot))}_{L^{p}(t_0, t_0+T'_{n}; X)} \\
    &\quad +\norm{ n^{\rho/p}g(nt)}_{L^{p}(t_0, t_0+T'_{n}; X)}\\
    &\lesssim  \norm{ u^{n}_{0}- u^{0}_{0}}_{X_{p}} + \norm{ \phi_{n} - \phi_{0}}_{L^{p}(t_0, t_0+T'_{n}; X_p)} \\
    &\quad + \norm{ n^{\rho/p}g(nt)}_{L^{p}(t_0, t_0+T'_{n}; X)}.
    \end{aligned}
\end{equation}
{ We know that the following estimate holds
\begin{align}\label{embedding1}
    \norm{ \phi_{n}(t) - \phi_{0}(t) }_{X_p} \lesssim \norm{ u^{n}_{0}- u^{0}_{0}}_{X_{p}}+\norm{ \phi_{n} - \phi_{0} }_{W^{1,p}(t_0,t_0+T'_n; X) \cap L^{p}(t_0,t_0+T'_n; D(A))},
\end{align}} for all $t \in [t_0, t_0+T'_n]$, see \cite[Corollary 1.14.(ii)]{interpolation_theory}. { Therefore, combining} { \eqref{embedding1} and \eqref{main_inrquality}}, we obtain the following estimate{
\begin{equation}
    \begin{aligned}
        \norm{ \phi_{n}(t) - \phi_{0}(t) }_{X_p}\lesssim  \norm{ u^{n}_{0}- u^{0}_{0}}_{X_{p}} + \norm{ \phi_{n} - \phi_{0}}_{L^{p}(t_0, t_0+T'_{n}; X_p)} + \norm{ n^{\rho/p}g(nt)}_{L^{p}(t_0, t_0+T'_{n}; X)}
    \end{aligned}
\end{equation}
for all $t \in [t_0, t_0+T'_n]$. Since for all \( a, b \geq 0 \) and \( p > 1 \), the inequality 
$(a + b)^p \leq 2^{p-1}(a^p + b^p)$ holds, we deduce %the following estimate holds
\[
\begin{aligned}
    \| \phi_n(t) - \phi_0(t) \|_{X_p}^p \lesssim & \, \| u_0^n - u_0^0 \|_{X_p}^p + \| \phi_n - \phi_0 \|_{L^p(t_0, t_0 + T'_n; X_p)}^p
     \, + \norm{n^{\rho/p} g(nt) }_{L^p(t_0, t_0 + T'_n; X)}^p
\end{aligned}
\]% Now, by using the Gronwall's inequality (Since, $ [t_0+T'_{n}] \ni t \mapsto \norm{ \phi_{n}(t) - \phi_{0}(t) }_{X_P}^p$ is a continuous function) we can conclude that for all  $t \in [0, T'_n]$, 
for all \( t \in [t_0, t_0+T'_n] \).
Note that the map $\norm{ \phi_{n}(\cdot) - \phi_{0}(\cdot) }^p_{X_p}$ in the previous inequality defines a continuous function on \( [t_0, t_0+ T_n'] \). By Gronwall's inequality,  we conclude that, for all \( t \in [t_0, t_0+ T_n'] \),  the following estimate holds %we have 
\begin{align}
        \norm{ \phi_{n}(t) - \phi_{0}(t) }_{X_p}^p &\lesssim  \norm{ u^{n}_{0}- u^{0}_{0}}^p_{X_{p}}+ \norm{ n^{\rho/p}g(nt)}^p_{L^{p}(t_0, t_0+T'_{n}; X)}\\
        &\lesssim \norm{ u^{n}_{0}- u^{0}_{0}}^p_{X_{p}}+ \norm{ n^{\rho/p}g(nt)}^p_{L^{p}(t_0, t_0+T'_{0}; X)}.\label{after_Gronwall}
\end{align}}
{Given that for all $T>0$, we have that
\[
\left( \frac{1}{nT} \right)^{1 - \rho} \int_{nt_0}^{nt_0+nT} \| g(s) \|_X^p \, ds \to 0 \quad \text{and} \quad \| u_0^n - u_0^0 \|_{X_p} \to 0 \text{ as } n \to \infty,
\]
and utilizing the continuity of \( \phi_n \), we deduce from \eqref{after_Gronwall} that \( T_n \geq T_n' = T_0' \). 

From \eqref{after_Gronwall} we also obtain 
\[
\sup_{[t_0, t_0 + T_0']}\| \phi_n(t) - \phi_0(t) \|_{X_p} \to 0 \quad \text{as} \quad n \to \infty
\]
%uniformly on the interval \( [t_0, t_0 + T_0'] \).
}
% Now by $\displaystyle\left(\frac{1}{nb-na}\right)^{1-\rho }\int_{na}^{nb} \norm{ g(s) }_X ^p ds \rightarrow 0$, $\norm{u_0^n-u_0^0}_{X_p}\to 0$ as $n\to\infty$ { and the continuity of $\phi_n$, we get that $T_n \geq T_n' \geq T'_0$, from \eqref{after_Gronwall} it follows that 
% \[
% \norm{ \phi_{n}(t) - \phi_{0}(t) }_{X_p} \rightarrow 0\quad \text{as $n \rightarrow \infty$}
% \]
% uniformly on $[t_0,t_0+T'_0]$. 
and it is clear from \eqref{main_inrquality} that %as $n \rightarrow \infty$,
\begin{align*}
    \norm{ \phi_{n} - \phi_{0} }_{W^{1,p}(t_0,t_0+T'_0; X) \cap L^{p}(t_0,t_0+T'_0; D(A))} \rightarrow 0\quad \text{as $n \rightarrow \infty$}.
\end{align*}
{ { From the previous analysis, we can ensure that for all \( k \in \mathbb{N} \), there exists \( n_k \in \mathbb{N} \) such that \( T_0 - \frac{1}{k} \leq T_{n_k} \). Therefore, we obtain the inequality \eqref{T_0_for_2nd_thm_leq_limsup}
}}
\end{proof}

\section{ Navier-Stokes equations with singularly oscillating forces} \label{application}
In this section, we will apply Theorem \eqref{4.2} to prove  our third main result concerning the Navier-Stokes equations with singularly oscillating  external forces.\\

Let $\rho \in [0,1)$ be a fixed parameter, $p \in (1, \infty)$, and $\Omega \subset \R^d$, $d=2,3$, be a bounded domain with boundary $\partial \Omega $ of class $C^{3}$. We consider the three-dimensional Navier-Stokes equations with the no-slip boundary condition
{\begin{equation}\label{eq:motion}
\left\{
\begin{aligned}
   &\nabla\cdot v_n %\partial_{x_1}v_n^{1}+\partial_{x_2}v_n^2
   =0\qquad&&\text{in }(0,T)\times\Omega,
   \\
   &\partial_t v_n+ (v_n\cdot \nabla)v_n % v_n^1\partial_{x_1}v_n+v_n^2\partial_{x_2}v_n
   =\nu\Delta v_n-\nabla P+n^{\frac{\rho}{p}}g(nt,x) \qquad&&\text{in } (0,T)\times \Omega,
   \\
   &v_n=0\qquad&&\text{on }(0,T)\times \partial \Omega, 
\end{aligned}
\right.    
\end{equation}}
Our next objective is to rewrite \eqref{eq:motion} as the evolution equation \eqref{important} on a Banach space $X$. Set 
\[
L_\sigma^{p}(\Omega):=\{v\in L^{p}(\Omega):\; \div v=0\quad\text{in }\Omega,\ v\cdot \n=0\quad\text{on }\partial \Omega\}, 
\]
where the divergence condition holds in the sense of distributions, while the boundary condition holds in the {sense of weak derivatives}. Here, $\n$ denotes the unit, outward normal to $\Omega$. We then define
\[
X:=L^{p}(\Omega),
\]
with norm 
\[
\norm{u}_0:=\norm{u}_{L^{p}(\Omega)}. 
\]
When ${p}=2$, $X$ is a Hilbert space with the inner product
\[
\langle u_1,u_2\rangle:=\int_\Omega\, u_1\cdot u_2\; dx,
\]
and associated norm $\norm{u}_0$ defined above (with ${p}=2$). We also introduce the Banach space 
\[
X_1:=H^2_{p}(\Omega)\cap {_0}H^1_{p}(\Omega)\cap L_{\sigma}^{p}(\Omega), 
\]
and the the operators: 
\begin{equation}\label{eq:operators}
\begin{split}
& A:\; u\in D(A):=X_1\mapsto  Au:=-\nu\P\Delta \in X,
\\
& f_n:\; u\in D(f_n)\subset (0,\infty)\times X
\mapsto   f_n(t,u):=\P\left(-(v_n\cdot \nabla)v_n +n^{\rho/p}g(nt,x)\right) \in X,
\end{split}
\end{equation}
where $\P$ denotes the Helmholtz projection of $L^{p}(\Omega)$ onto $L_\sigma^{p}(\Omega)$.\\

 We assume that the external force $g \in L^p_{\mathrm{loc}}(0,\infty; X)$ satisfies 
\[
\left(\frac{1}{T}\right)^{1-\rho }\int_0^T \norm{g(s)}^p_{X} ds \rightarrow 0,\qquad\text{  as $T \rightarrow \infty$}.
\]
We denote the space of such forces by $ L^p_{\text{avr}}(0, \infty; X) $. The final main result of our paper reads as follows. 
\begin{theorem}\label{main theorem3}
    Let $p > \frac{5}{2}$, and consider a solution $v\in W^{1,p}(0,{T}; X) \cap L^{p}(0,{T}; D(A))$ of the homogeneous Navier-Stokes equations: 
    %the following equations:
    \begin{equation}
\left\{
\begin{aligned}
   &\nabla\cdot v %\partial_{x_1}v_n^{1}+\partial_{x_2}v_n^2
   =0\qquad&&\text{in }(0,T)\times\Omega,
   \\
   &\partial_t v+ (v\cdot \nabla)v % v^1\partial_{x_1}v+v^2\partial_{x_2}v
   =\nu\Delta v-\nabla P \qquad&&\text{in } (0,T)\times \Omega,
   \\
   &v=0\qquad&&\text{on }(0,T)\times \partial \Omega, 
\end{aligned}
\right.    
\end{equation}
with $v(0) \in X_p$. Then, for all $g \in$ \(L^p_{\mathrm{avr}}(0, \infty; X) \) and $v_n(0)=v(0)$, the sequence $(v_n)_{n\ge 1}$ of solutions of \eqref{eq:motion} satisfies 
\[
\sup_{0\leq t\leq {T'}}\|v(t)-v_n(t)\|_{X_p}\to 0 \quad\text{as }n\to\infty
\]
%uniformly in $[0,T]$, 
for all ${T'<T}$.
%*****
%    There exists $p>1$ such that for all $g \in$ \(L^p_{\mathrm{avr}}(0, \infty; X) \) the corresponding solution of \eqref{eq:motion} converges uniformly in $X_p$ on the compact sub-interval to the solution of the following equations:
%    \begin{equation}
%\left\{
%\begin{aligned}
%   &\div v=0\qquad&&\text{in }\Omega\times(0,T),
%   \\
%   &\partial_t v+ v_i\partial_{x_i}v=\nu\Delta v-\nabla p \qquad&&\text{in }\Omega\times(0,T),
%   \\
%   &v=0\qquad&&\text{on }\partial \Omega\times (0,T). 
%\end{aligned}
%\right.    
%\end{equation}

Moreover, $(v_n)_{n\ge 1}$ converges to $v$ also in $W^{1,p}(0,{T'}; X) \cap L^{p}(0,{T'}; D(A))$, for all ${T'<T}$.
\end{theorem}
\begin{proof}

From \cite[Theorem 3.9.]{MR3916775},  it can be deduced that for all \( p > 1 \), \( A \) possesses the property of maximal \( L^p \)-regularity. In addition, the mapping \( v \mapsto(v\cdot \nabla)v: W^{s,p}(\Omega) \rightarrow L^p(\Omega) \)  is defined and bilinear whenever $s > \max\{1,\frac{3}{p}\}$. If we choose $p>\frac{5}{2}$, then we have the mapping \( v \mapsto(v\cdot \nabla)v: X_{p} \subset W^{2-\frac{2}{p},p}(\Omega) \rightarrow L^p(\Omega) \)  is locally Lipschitz. The proof of the theorem is then an immediate consequence of Theorem \eqref{4.2}.
\end{proof}
% \begin{remark}
% The previous result still holds for all \( p > \frac{5}{2} \).  In fact, $X_{p}\hookrightarrow X_{\frac{5}{2}}$ for all \( p > \frac{5}{2} \) (this follows from \cite[Proposition 1.3]{interpolation_theory} and \cite[Proposition 1.4]{interpolation_theory}). 
% %This is due to some embedding theorems for real interpolation spaces (see, \cite[Proposition 1.3]{interpolation_theory} and \cite[Proposition 1.4]{interpolation_theory}). 
% As a consequence, the map \( v \mapsto (v\cdot \nabla)v \) is still a locally Lipschitz function from \( X_p \) to \( X \), for all \( p > \frac{5}{2} \).
% %****
% %    Although we have proven here that there exists a \( p \) such that the result holds, the embedding theorem of real interpolation spaces further justifies that there exists \( p_0 > 1 \) such that for all \( p > p_0 \), the map \( v \mapsto v_i \partial_{x_i} v \) is a locally Lipschitz function from \( X_p \) to \( X \). Therefore, the previous result holds for all \( p > p_0 \). For a reference to the embedding theorem, one can consult \cite[Proposition 1.3]{interpolation_theory} and \cite[Proposition 1.4]{interpolation_theory}.
% \end{remark}

\section{Acknowledgments}
 The author wishes to express sincere gratitude to Prof. Francesco Cellarosi and Prof. Giusy Mazzone for their unwavering support, expert guidance, and continuous encouragement throughout the course of this project. Their insightful feedback and valuable suggestions have been instrumental in shaping the direction and quality of this work.

Heartfelt thanks are also extended to Prof. Simonett for the enlightening discussions and intellectual engagement that greatly enriched this research.

The author gratefully acknowledges the support of the Natural Sciences and Engineering Research Council of Canada (NSERC).

% Finally, the author would like to thank the anonymous referees for their valuable comments and suggestions.
% It may be of interest to the following editors

\bibliographystyle{plain}
\bibliography{Cellarosi_Dutta_Mazzone-Spectral_Stability}

\end{document}